\documentclass[11pt,twoside]{amsart}

\usepackage{latexsym}
\usepackage{epsfig}
\usepackage{amsmath}
\usepackage{amsthm}
\usepackage{amsfonts}
\usepackage{exscale}
\usepackage{graphicx}
\usepackage{color}

\newtheorem{cor}{Corollary}
\newtheorem{prop}{Proposition}

\newtheorem{remark}{Remark}
\newtheorem{lemma}{Lemma}
\newtheorem{theorem}{Theorem}

\newtheorem{assumption}{Assumption}

\newcommand{\sm}{\left(\begin{smallmatrix}}
\newcommand{\esm}{\end{smallmatrix}\right)}

\newcommand{\la}{\lambda}

\newcommand{\codim}{\operatorname{codim}}

\newcommand{\eps}{\varepsilon}

\newcommand{\Bt}{\widetilde{B}}

\newcommand{\C}{\ensuremath{\mathbb{C}}}
\newcommand{\R}{\ensuremath{\mathbb{R}}}
\newcommand{\N}{\ensuremath{\mathbb{N}}}
\newcommand{\Z}{\ensuremath{\mathbb{Z}}}

\newcommand{\norm}[1]{\left\Vert#1\right\Vert}

\newcommand{\bk}{\mathbf{k}}
\newcommand{\bx}{\mathbf{x}}

\newcommand{\be}{\begin{equation*}}
\newcommand{\ee}{\end{equation*}}
\newcommand{\bea}{\begin{eqnarray*}}
\newcommand{\eea}{\end{eqnarray*}}
\newcommand{\ben}{\begin{eqnarray}}
\newcommand{\een}{\end{eqnarray}}
\newcommand{\beq}{\begin{equation}}
\newcommand{\eeq}{\end{equation}}
\newcommand{\enq}{\end{equation}}
\def\llangle{\left\langle}\def\rrangle{\right\rangle}

\title{Spectrum created by line defects in periodic structures}

\author{B.M.Brown}
\address{Cardiff School of Computer Science,
Cardiff University, Cardiff, CF24 3AA, Wales, UK}
\email{Malcolm.Brown@cs.cardiff.ac.uk}
\author{V.Hoang}
\address{Institute for Analysis,
Karlsruhe Institute of Technology (KIT), Kaiserstrasse 89, Karlsruhe,
Germany}\email{ duy.hoang@kit.edu}
\author{M.Plum} \address{Institute for Analysis,
Karlsruhe Institute of Technology (KIT), Kaiserstrasse 89, Karlsruhe,
Germany}\email{michael.plum@kit.edu}
\author{I.Wood} \address{School of Mathematics, Statistics and Actuarial Sciences,
 University of Kent, Canterbury, CT2 7NF, UK}\email{ i.wood@kent.ac.uk}

\begin{document}

\maketitle
 
\section{Introduction}

The spectrum of periodic differential operators typically exhibits a band-gap structure.
In this paper, we will consider  perturbations to periodic differential operators and investigate the spectrum the perturbation induces in the gaps. More specifically, we consider the operator
$$
L_0 =-\frac{1}{\eps_0(x,y,z)}\Delta
$$
in $\R^3$ with $\eps_0$ periodic in all three directions. The perturbation is introduced by replacing $\eps_0$ by $\eps_0+\eps_1$ where we assume that $\eps_1$ is still periodic in one direction, but compactly supported in the remaining two directions, creating a line defect. We will show that even small perturbations $\eps_1$ lead to additional spectrum in the spectral gaps of the unperturbed operator $L_0$ and investigate some properties of the spectrum that is created.

Vector-valued equations of this form, coupled by a curl-condition (see, e.g.~\cite{Fey}), arise in the study of elastic wave propagation in phononic crystals. The present paper is a first step in the investigation of these problems.
The analogous problem in two dimensions has been well investigated in the literature due to its connection with photonic and phononic waveguides
(see, e.g.~\cite{AS, BHPW, BHPW12, KuchWav1}). 
In the case of a line defect in two dimensions the analysis involves studying the band functions arising from the Floquet-Bloch decomposition as functions of one complex variable in which they are analytic. This was of great help in \cite{BHPW12}. In the three dimensional situation this analyticity is no longer at hand. This poses new mathematical challenges which are addressed here.

We remark that some similar spectral results for the Schr\"odinger equation have been obtained by a different method in \cite{PLAJ,Prodan}.
Moreover, perturbations of periodic structures  for the  3D Maxwell equations are studied in \cite{KuchWav2, MiaoMa, MiaoMa2}.
An up-to-date account of the theory of photonic crystals is given in \cite{Joann}.

\section{Mathematical setting}

\subsection{Definition of the operators}
We consider the spectral problem for the weighted Laplacian in three-dimensional space. 
Let $\eps_0\in L^\infty(\R^3)$ be periodic in all three directions and
 positively bounded away from zero. Without loss of generality, we
take the unit cube $[0, 1]^3$ to be the  cell of periodicity. In the weighted space $L^2_{\eps_0}(\R^3)$, we
introduce the self-adjoint operator $L_0$, formally given by the expression
$$
L_0 =-\frac{1}{\eps_0(x,y,z)}\Delta.
$$
$L_0$ is rigorously defined as the unique unbounded, self-adjoint operator on $L^2_{\eps_0}(\R^3)$ associated with the bilinear form 
\beq\label{defb}
b(u, v) = \int_{\R^3} \nabla u(\bx)  \overline{\nabla v(\bx)}~d\bx, \quad (u, v\in H^1(\R^3)),
\eeq
via the standard representation theorem for bilinear forms (see e.g. \cite{Kato}). As a result, $D(L_0)=H^2(\R^3)$.

We introduce the perturbation in the form $\eps = \eps_0+\eps_1$, where
$\eps_1$ is supported in $W=\R\times(0,1)^2$ and is periodic with period 1 in the $x$-direction.
To ensure that the perturbation is non-trivial,
we assume that $\eps_1\geq 0$ and is bounded away from $0$ on some ball $D\subset (0,1)^3$.

\begin{remark} Similar results to the ones we prove in this paper hold under the assumption $\eps_1\leq 0$ and is bounded away from $0$ on some ball $D$ provided 
$$\operatorname{essinf}_{\R^3}(\eps_0+\eps_1)>0.$$
\end{remark}

Now, the operator
$$
L =-\frac{1}{\eps(x,y,z)}\Delta
$$
is defined as the unique self-adjoint operator associated to the form \eqref{defb} in
the weighted space $L^2_\eps(\R^3)$, whence  $D(L)=H^2(\R^3)$.

The analysis of the operators $L_0$ and $L$ relies on the Floquet-Bloch or Gel'fand transform (see \cite{KuchmentBook}).
The (partial) Floquet-Bloch transform in the $x$-direction is defined by
\beq\label{floqx}
(U_x f)(\bx, k) = \frac{1}{\sqrt{2 \pi}} \sum_{m\in \Z} e^{i k m} f(\bx - m{\bf e}_x),\quad   \bx \in \Omega=(0,1)\times\R^2,\;\; k\in [-\pi,\pi],
\eeq
where ${\bf e}_x$ is the unit vector in the $x$-direction.
$U_x: L^2(\R^3) \rightarrow L^2((-\pi, \pi), L^2(\Omega))$ is an isometry. The investigation
of the spectra of $L_0$ and $L$ reduces in a standard way to the investigation of the spectra of
  self-adjoint operators $L_0(k_x), L(k_x)$, $k_x\in [-\pi, \pi]$ which  are
formally given by
\beq 
L_0(k_x)u=-\frac{1}{\eps_0(x,y,z)}\Delta u \quad \hbox{and}
\quad L(k_x)u=-\frac{1}{\eps(x,y,z)}\Delta u
\eeq
acting on functions satisfying $k_x$-quasiperiodic boundary conditions in the $x$-direction, i.e.
$$
u(\bx + m {\bf e}_x) = e^{i k_x m} u(\bx), \quad \bx\in \R^3, m\in \Z.
$$
$L_0(k_x)$ and $L(k_x)$ can be defined rigorously via quadratic forms, similarly to $L_0$ and $L$ above.

From now on we will fix $k_x \in [-\pi,\pi]$.
We wish to compare the spectra of $L_0(k_x)$ and $L(k_x)$.
As $\eps_1\vert_\Omega$ is compactly supported,
the essential spectra coincide and we only need to look for those eigenvalues of $L(k_x)$
which the perturbation may have introduced in the gaps of the essential spectrum.

Note that the background problem is additionally periodic in the $y$- and $z$-directions.
Hence we can take the Floquet-Bloch transform in both of these directions to get
operators $ L_0(k_x,k_y,k_z)$ acting on $(0,1)^3$ with quasi-periodic boundary conditions in all three directions.
These operators have eigenvalues which we denote by $\lambda_s(\bf{k})$ and corresponding eigenfunctions $\psi_s(\bf{k})$ which we refer to as the Bloch functions, where ${\bf k}=(k_y,k_z)$.

For a more detailed discussion of the contents of this section, we refer to \cite{BHPW12} and references therein.

\subsection{A Birman-Schwinger-like formulation} Our arguments rest on a Birman-Schwinger-like
formulation of the eigenvalue problem, which we now introduce.

Let $(\mu_0,\mu_1)$ be a gap of the spectrum of the operator $L_0(k_x)$, where $\mu_0$ and $\mu_1$ are assumed to lie at the end of spectral bands.
Consider the eigenvalue problem for $L(k_x)$,
\beq\label{eq:pert} -\Delta u = \lambda (\eps_0 + \eps_1) u\quad\hbox{ on } \Omega= (0,1)\times \R^2   \eeq
where $\lambda\in (\mu_0,\mu_1)$.

Let $\Lambda_0$ denote the base of the   spectrum of $L_0(k_x)$. We note that $\Lambda_0\geq 0$ and that $\sigma(L(k_x))\subseteq [0,\infty)$. In the case when $\Lambda_0>0$, our analysis also includes the gap $(-\infty, \Lambda_0)$, provided we assume $\lambda>0$.

From the equation \eqref{eq:pert} we get, for eigenpairs $(u,\la)$,
$$(L_0(k_x)-\lambda)u = -\frac{1}{\eps_0} \Delta u -\lambda u = \lambda \frac{\eps_1}{\eps_0} u.$$
It then follows that $\lambda$ is an eigenvalue in the gap iff 
\beq u=\lambda \left(L_0(k_x) -\lambda\right)^{-1} \left( \frac{\eps_1}{\eps_0}u\right) \label{FP} \eeq
holds (on $\Omega$)
for some non-zero $u$ satisfying quasi-periodic boundary conditions in the $x$-direction.

Given a solution $u$ of \eqref{FP}, set $v:=\sqrt{\frac{\eps_1}{\eps_0}}u.$
Then $v$ is supported in $[0,1]^3$, as $\eps_1|_\Omega$ is, and $v$ satisfies
\beq\label{eq:v2}
v=\lambda\sqrt{\frac{\eps_1}{\eps_0}}\left(L_0(k_x)-\lambda\right)^{-1}\sqrt{\frac{\eps_1}{\eps_0}}v.
\eeq
Conversely, if $v$ satisfies \eqref{eq:v2}, then
$$u:=\lambda\left(L_0(k_x)-\lambda\right)^{-1}\sqrt{\frac{\eps_1}{\eps_0}}v$$
satisfies \eqref{FP}
and lies in $L^2_{\eps_0}(\Omega)\cap {\rm dom} (L_0(k_x))$. Hence,  $L(k_x)u=\lambda u$. It is therefore sufficient for our purposes to study \eqref{eq:v2}.

We now define the operator $A_\lambda$ on $L^2_{\eps_0}({(0,1)^3})$ by
$$
A_\lambda v=\left( \lambda\sqrt{\frac{\eps_1}{\eps_0}}\left(L_0(k_x)-\lambda\right)^{-1}\sqrt{\frac{\eps_1}{\eps_0}} v\right)\Bigg\vert_{(0,1)^3}
$$
and note that \eqref{eq:v2} has a non-trivial solution if and only if $1$ is an eigenvalue of the operator $A_\lambda$.

\section{Preliminary results}
We denote by $\langle \cdot, \cdot \rangle_{\eps_0}$ the inner-product in $L^2_{\eps_0}((0,1)^3)$.
Using the representation of the Floquet-Bloch transform $U$ in the $y$- and $z$-directions in terms of the Bloch functions and the fact that
$Ur(\cdot , {\bf k})=\frac1{2\pi} r(\cdot)$ for any function $r$ such that $r=0$ outside $(0,1)^3$, 
we  have the following resolvent formula \cite{KuchmentBook, OdehKeller, RS} for the operator
$L_0(k_x)$:
\beq
(L_0(k_x) -\la)^{-1}r=\frac{1}{(2\pi)^2} \sum_{s \in \N} \int_{-\pi}^\pi\int_{-\pi}^\pi(\lambda_s({\bf k})-\la)^{-1} \llangle r,\psi_s(\cdot,{\bf k})\rrangle_{\eps_0}\psi_s(\cdot,{\bf k}) d{\bf k}
\label{eq:res1}
\enq
for $\la$ outside the spectrum of $L_0(k_x)$ and $r\in L^2((0,1)^3)$ extended by  $0$ on $\Omega \setminus (0,1)^3$.

  \begin{lemma} Let $\la\in\R$. Then
$  A_{\la} : L^2_{\eps_0}({(0,1)^3}) \to L^2_{\eps_0}({(0,1)^3}) $ is symmetric and compact.
  \end{lemma}

\begin{proof} Let $u,v\in L^2_{\eps_0}({(0,1)^3})$. Then
  \begin{eqnarray*}
  \llangle \eps_0 A_\la u,v\rrangle_{L^2({(0,1)^3})} &=& \llangle \eps_0 \la  \left (L_0(k_x) -\la\right )^{-1} \sqrt{\frac{\eps_1}{\eps_0}}u, \sqrt{\frac{\eps_1}{\eps_0}} v\rrangle_{L^2(\Omega)}
   \\
  &=&\llangle \eps_0\sqrt{\frac{\eps_1}{\eps_0}} u,\la  \left ( L_0(k_x) -\la\right )^{-1} \sqrt{\frac{\eps_1}{\eps_0}} v \rrangle_{L^2(\Omega)} \\
  &=&\llangle \eps_0u,\la\sqrt{\frac{\eps_1}{\eps_0}} \left ( \left ( L_0(k_x)-\la\right )^{-1} \sqrt{\frac{\eps_1}{\eps_0}} v\right ) \Bigg \vert_{(0,1)^3} \rrangle_{L^2({(0,1)^3})}\\
   \ &=&\ \llangle \eps_0 u, A_{\la} v\rrangle_{L^2((0,1)^3)},
  \end{eqnarray*}
  so the operator is symmetric. Moreover, by standard estimates (see e.g. \cite[p.71]{Agmon}) for the elliptic operator on the slab $\Omega$,
  \bea
  \norm{ \left ( L_0(k_x) -\la \right )^{-1}\sqrt{\frac{\eps_1}{\eps_0}} u}_{H^1({(0,1)^3})}
  &\leq   \norm{ \left (L_0(k_x) -\la \right )^{-1} \sqrt{\frac{\eps_1}{\eps_0}} u }_{H^1(\Omega)}\\
  &\leq C_{\la} \norm{  u }_{L^2(\Omega)}=C_{\la}\norm{ u}_{L^2((0,1)^3)}.
  \eea
  Thus $A_\la $ is the composition of a compact map with the continuous map of multiplication by the function $\sqrt{\frac{\eps_1}{\eps_0}}$ and
  is therefore compact as a map from $L^2({(0,1)^3}) \to L^2({(0,1)^3})$. Multiplication by the bounded and boundedly invertible weight $\eps_0$ does not change this.\end{proof}


We now investigate the dependence of the positive eigenvalues  of $A_\la$ on $\la$. 
We have the following standard variational characterisations of the $n$-th highest   positive eigenvalue (if it exists):
\beq\kappa_n(\la)=\min_{\codim L = n-1}\ \sup_{u\in L} \frac{ \llangle A_\la u,u\rrangle_{\eps_0}}{\llangle u,u\rrangle_{\eps_0}}
= \max_{\dim L =n} \ \min_{u\in L} \frac{ \llangle A_\la u,u\rrangle_{\eps_0}}{\llangle u,u\rrangle_{\eps_0}}.\label{varchar}\eeq

%

\begin{lemma}\label{lem:cont} For $\la$ in the spectral gap $(\mu_0,\mu_1)$ we have that
$\la\mapsto \kappa_n(\la)$ is continuous and increasing.
\end{lemma}

\begin{proof}
As $\la\mapsto A_\la$ is norm-continuous, we have that  for $\la \in (\mu_0, \mu_1)$  and   $\widetilde{\eps}>0$, there exists $\delta>0$ such that $|\la-\widetilde{\la}|<\delta$ implies,
for every $u$, 
$$\left|\llangle A_\la u, u \rrangle_{\eps_0}- \llangle A_{\widetilde{\la}} u, u \rrangle_{\eps_0}\right| \leq \widetilde{\eps} \norm{u}^2_{\eps_0}.$$ 
Thus
$$\frac{\llangle A_{\widetilde{\la}} u, u \rrangle_{\eps_0}}{\norm{u}^2_{\eps_0}}  \leq  \frac{\llangle A_{\la} u, u \rrangle_{\eps_0}}{\norm{u}^2_{\eps_0}}+\widetilde{\eps},$$
and therefore $\kappa_n(\widetilde \la) \leq \kappa_n(\la)+\widetilde{\eps}$  by \eqref{varchar}.
Similarly, we obtain the reverse inequality. Together these imply continuity of $\la\mapsto \kappa_n(\la)$.

  Let $\mu_0<\la<\widetilde \la<\mu_1$.
  Then
   $$\frac{\widetilde \la}{\la_s({\bf k})-\widetilde \la}-\frac\la{\la_s({\bf k})-\la}=\frac{ (\widetilde \la-\la)\la_s({\bf k})}{(\la_s({\bf k})-\widetilde \la)(\la_s({\bf k})-\la)} \geq 0
   $$
   since $ (\la_s({\bf k})-\widetilde \la) (\la_s({\bf k})-\la)>0$ and $\la_s({\bf k})\geq 0$ for all $s$ and all ${\bf k}$.   Thus, by \eqref{eq:res1}  $\la\mapsto \kappa_{n}(\la)$ is monotonically increasing.
\end{proof}

We now introduce some notation and a basic assumption which will remain valid for the rest of the paper. Let
\bea
&\Sigma=\{ (s,{\bf k}) : \la_{s}({\bf k})=\mu_1\},\\
&S_{\bf k}=\{s:(s,{\bf k})\in \Sigma\}, \\
&S=\{ s :\text{   there\;is\;a\;} {\bf k}\;  \text{with \;} (s,{\bf k})\in \Sigma \} = \bigcup_{\bf k} S_{\bf k},
\eea
where  $\mu_1$ is the upper end of the spectral gap under investigation. Lemma \ref{lem:5} below shows that under our following assumption the set $\Sigma$ is finite and we will denote its elements in the following by $(s_j, {\bf k}^j), j=1,...,n.$

\begin{assumption}\label{as1}
We assume that for all $(\hat s, \hat k) \in \Sigma$ there are $\alpha_{(\hat s, \hat k)}>0$   and $\delta_{(\hat s, \hat k)}>0$ such that  for all    ${\bf k }\in [-\pi,\pi]^2 $     satisfying  $|{\bf  k} -{\bf \hat k}| \leq \delta_{(\hat s,\hat {\bf k})}$ we have  $\la_{\hat s} ({\bf k})\geq \mu_1+\alpha_{(\hat s, \hat k)} | {\bf k}-{\bf \hat k}|^2$.
\end{assumption}

\begin{lemma}
For ${\bf k}\in [-\pi,\pi]^2,\;\la_s( {\bf k})\to \infty$ as $s\to\infty$.
\end{lemma}
\begin{proof}
This is an application of the min max characterisation (\ref{varchar}) of eigenvalues  of $L_0(k_x,k_y,k_z)$  to show that  $\min_{ {\bf k} \in [-\pi,\pi]^2} \la_s( {\bf k})$ is bounded below by the $s^{th}$-Neumann eigenvalue of $-\eps_0^{-1}\Delta$ on the cube  and hence tends to infinity as $s\to\infty$.
\end{proof}
\begin{cor} \label{cor:1}
There is an $s_0\in \N$ such that for all $s\geq s_0$ and for all ${\bf k} \in [-\pi,\pi]^2$ we have $\la_s({\bf k})\geq \mu_1+1$.
\end{cor}


\begin{lemma}\label{lem:5}
The set $\Sigma$ is both finite and isolated in the sense that   there is a $\delta >0$ such that for all $s\not \in S$  we have   $| \la_s({\bf k})-\mu_1 |\geq \delta$  for all  ${\bf k}\in [-\pi,\pi]^2$.
\end{lemma}
\begin{proof}
The proof is presented in two parts.
We first show $\Sigma$ is finite.  To do this we shall assume the contrary.  This implies there is an  injective sequence $(s_n,{\bf k}_n)\in \N \times [-\pi,\pi]^2 $ such that $\la_{s_n}({\bf k}_n)=\mu_1$.  In particular  this means that $s_n\leq s_0$
(with $s_0$ from Corollary \ref{cor:1}),
implying that  there is $\hat s \in \{1,\ldots, s_0\}$  and a subsequence $(s_{n_j})$ such that $(s_{n_j})=\hat s$  for all  $j$. Thus, the subsequence $({ \bf k}_{n_j})$ must be injective.

 Since ${ \bf k}_{n_j} \in [-\pi,\pi]^2$, we have that  ${ \bf k}_{n_{j_\nu}} \to \hat  { \bf k} $ for some subsequence ${ \bf k}_{n_{j_\nu}}$ and some  $\hat { \bf k}\in [-\pi,\pi]^2$, and as  $\la_{\hat s}$ is continuous this implies that $\la_{\hat s}({\bf k}_{n_{j_\nu}})\to \la_{\hat s} (\hat {\bf k})$.
 Also $\la_{{\hat s}}({ \bf k}_{n_{j_\nu}})=\la_{s_{n_{j_\nu}}}({ \bf k}_{n_{j_\nu}})=\mu_1$, so we get   $\la_{{\hat s}}(\hat {\bf k})=\mu_1$, i.e. $(\hat s,\hat {\bf k})\in \Sigma$.  Using Assumption \ref{as1}, this shows
 $$
 \mu_1=\la_{\hat s} ( {\bf k}_{n_{j_\nu}})\geq \mu_1 + \alpha_{(\hat s,\hat {\bf k})} | {\bf k}_{n_{j_\nu}}-\hat {\bf k}|^2
$$
for sufficiently large $\nu$, which provides the contradiction since $({ \bf k}_{ n_{j_\nu}})$ is injective.

We now prove  that $\Sigma$ is isolated.
For each $s\not \in S$ such that $s\leq s_0$, we have   that $|\la_s({\bf k})-\mu_1 | >0$ for all ${\bf k}\in [-\pi,\pi]^2$.
 Since $\la_s$ is continuous, this implies the existence of some $\delta_s>0$ such that   $|\la_s({\bf k})-\mu_1|\geq \delta_s>0$.
Now
$$\delta:= \min  \{\delta_s:s\not\in S,\; s\leq s_0  \}>0.$$
So
$ | \la_s({\bf k})-\mu_1|\geq \delta$ for all $s\not  \in S,\;s\leq s_0$, and  ${\bf k}\in [-\pi,\pi]^2$.
Moreover, for all $s \geq s_0, \; |\la_s({\bf k})-\mu_1|\geq 1$  for all ${\bf k} \in [-\pi,\pi]^2$, by Corollary
\ref{cor:1}.
Together, for each $s\not \in S$, we have  $ | \la_s({\bf k}) -\mu_1 | \geq \min \{\delta,1\}$.
\end{proof}

We remark that in the case of planar wave-guides, where $ k\in [-\pi, \pi]$,   the  finiteness of $\Sigma$ follows also from the analyticity of the band functions
$\la_s(k)$, and the Thomas argument excluding constant band functions (see, e.g.~\cite{BHPW12}).

We next prove a result on the linear independence of the Bloch functions:

\begin{prop}\label{lindep}   The set $\{\psi_{s_j}(\cdot,{\bf k}^j): j=1,...,n\}$ is linearly independent over the ball $D $, the set  where $\eps_1$ is bounded away from $0$.
\end{prop}

\begin{proof} Denote $\psi_{s_j}(\cdot,{\bf k}^j)$ by $\psi_j(\cdot)$ and
suppose $\sum_{j=1}^n \alpha_j\psi_j=0$ on $D$ for some $\alpha_1,...,\alpha_n\in\C$. As  $\lambda_{s_j}({\bf k}^j)=\mu_1$, we have
$$ \left(L_0(k_x)-\mu_1\right)\left(\sum_{j=1}^n\alpha_j\psi_j\right)=0 \quad\hbox{ on } \Omega. $$
As $D$ has interior points, unique continuation implies that $\sum_{j=1}^n \alpha_j\psi_j=0$ on $\Omega$.
Hence, for all ${\bf m}=(m_y,m_z)\in\N_0^2$ we have
$$ 0 = \sum_{j=1}^n \alpha_j\psi_j(x, y+m_y,z+m_z) = \sum_{j=1}^n \alpha_j e^{i{\bf k}^j{\bf m}}\psi_j(x, y,z).   $$
We partition the $k_y^j$ into $R_y$ groups, denoting $I_{k^{(r_y)}}=\{ j: \ k_y^j=k^{(r_y)}\}$ for $r_y=1,\dots , R_y$, such that $k^{(r_y)}\neq k^{(\tilde r_y)}$ and $| k^{(r_y)}-k^{(\tilde r_y)}| \neq 2 \pi$ for $r_y\neq \tilde r_y$, and obtain
$$\sum_{r_y=1}^{R_y}  e^{ik^{(r_y)}m_y}\sum_{j\in I_{k^{(r_y)}}} \alpha_j\psi_j=0\quad\hbox{ for any } m_y\in\N_0.   $$
Setting $\phi_{r_y}=\sum_{j\in I_{k^{(r_y)}}} \alpha_j\psi_j$ and $q_{r_y}=e^{ik^{(r_y)}}$, we get
\beq
\left(\begin{array}{ccc} 1  & \cdots  &1 \\ q_1&&q_{R_y}\\ \vdots&\ddots&\vdots\\ q_1^{{R_y}-1} & \cdots & q_{R_y}^{{R_y}-1}\end{array}\right)
\left(\begin{array}{c} \phi_1 \\  \vdots \\ \vdots\\ \phi_{R_y}\end{array}\right)=0.
\eeq
As the $q_j$ are distinct, this Vandermonde matrix is invertible and we get $\phi_{r_y}=0$ for all $r_y$.

Next, for fixed $r_y$, partition the $k_z^j$ with $j\in I_{k^{(r_y)}}$ into $R_z$ groups, denoting $I_{k^{(r_y),(r_z)}}=\{ j\in I_{k^{(r_y)}}: \ k_z^j=k^{(r_z)}\}$ and obtain, since $\phi_{r_y}=0$,
$$\sum_{r_z=1}^{R_z}  e^{ik^{(r_z)}m_z}\sum_{j\in I_{k^{(r_y),(r_z)}}} \alpha_j\psi_j=0 \quad\hbox{ for any } m_z\in\N_0.   $$
Setting $\varphi_{r_z}=\sum_{j\in I_{k^{(r_y),(r_z)}}} \alpha_j\psi_j$, the same argument as above yields $\varphi_{r_z}=0$ for all $r_z$.

As the $s_j$ for different $j\in I_{k^{(r_y),(r_z)}}$ are different, and the $\psi_{s_j}(\cdot,k^{(r_y)},k^{({r_z})})$ for different $s_j$ are linearly independent, this implies $\alpha_j=0$ for all $j$.
\end{proof}

For $r,u \in L^2_{\eps_0} ((0,1)^3)$ we define the following quantities, which will be useful later:
\ben
P(\widetilde {\bf k}) r&:=&\sum_{j=1}^n  \sum_{s\in S_{{\bf k}^j}} \langle r,\psi_s(\cdot, {\bf k}^j+\widetilde {\bf k})\rangle_{\eps_0} \psi_s(\cdot, {\bf k}^j+\widetilde {\bf k}),
\label{star}\\ \label{star2}
 f(  {\bf \widetilde k},u)&:=& \sum_{j=1}^n
   \sum_{s\in S_{{\bf k}^j}}   \left| \llangle  \sqrt{\frac{\eps_1}{\eps_0}}u,\psi_{s}(\cdot,{\bf k}^j+\widetilde {\bf k})\rrangle_{\eps_0}  \right|^2.
\een

\begin{remark} It is possible for one $s$ to be in several of the sets $S_{\bf k}$,
so   $P(\widetilde{\bf  k})$ is not necessarily  a projection.
\end{remark}

\begin{lemma}  \label{lem5} 
$P(\widetilde{{\bf k}})$ is Lipschitz continuous in $\widetilde{{\bf k}}$ near $0$.
\end{lemma}
\begin{proof}
We note that for each $j$ and small $\widetilde \bk$,
$$r \mapsto \sum_{s\in S_{{\bf k}^j}} \langle r,\psi_s(\cdot, {\bf k}^j+\widetilde {\bf k})\rangle \psi_s(\cdot, {\bf k}^j+\widetilde {\bf k})$$
is exactly the total projection on the eigenspaces associated with the eigenvalues $\lambda_s(\bk^j+\widetilde \bk)$, $ s\in S_{{\bf k}^j}$.
These eigenvalues equal $\mu_1$ for $\widetilde \bk = 0$,  and  for small $\widetilde \bk$  they lie close to each other and distant  from   any   $\la_s({\bf k}^j+\tilde {\bf k} )$ with $s\not \in S_{{\bf k}^j}$  (by Corollary \ref{cor:1}).
Therefore, the total projection can be written as a Riesz projection depending analytically on $\widetilde \bk$, since $L_0(k_x, {\bf k})$ does.
\end{proof}

\begin{cor} \label{cor}
For any $u$, we have that $f(\widetilde{{\bf k}},u)$ is Lipschitz continuous in $\widetilde{{\bf k}}$ near $0$, with a Lipschitz constant  $M\norm{\eps_1}_\infty\norm{u}^2_{\eps_0}$.
\end{cor}

\begin{proof}
  This follows immediately from Lemma \ref{lem5} and the representation
  $$f(  {\bf \widetilde k},u)=  \llangle  P(\widetilde {\bf k})  \sqrt{\frac{\eps_1}{\eps_0}}u, \sqrt{\frac{\eps_1}{\eps_0}}u  \rrangle_{\eps_0}. $$
 \end{proof}

\section{Main results}
In this section we prove results on the spectrum which the perturbation creates in $(\mu_0,\mu_1)$, the gap  of the spectrum of $L_0(k_x)$. We start with an upper bound on the number of eigenvalues created. We remind the reader that we assume throughout that $\lambda>0$.

\begin{theorem}\label{upper}
 There exists $c>0$
such that if $\norm{\eps_1}_\infty<c$, then the operator $L(k_x)$ has at most $n$ eigenvalues in the spectral gap $(\mu_0,\mu_1)$ of the operator $L_0(k_x)$.
\end{theorem}

\begin{proof}
We start by noting an equality for the Rayleigh quotient. Let $u\in L^2({(0,1)^3})$ and  extend $u$ by $0$ to $\Omega \setminus (0,1)^3$.
 Using \eqref{eq:res1} we have, for $\la\in (\mu_0,\mu_1)$,
  \begin{eqnarray}
  \llangle \eps_0A_\la u,u \rrangle_{L^2({(0,1)^3})}
   &=&\la \llangle \eps_0  \sqrt{\frac{\eps_1}{\eps_0}}\left [ \left ( L_0(k_x) -\la \right )^{-1}  \sqrt{\frac{\eps_1}{\eps_0}} u \right ]\Bigg\vert_{(0,1)^3}, u \rrangle_{L^2( (0,1)^3)} \nonumber \\
  & =& \la \llangle \eps_0   \left (  L_0(k_x)  -\la \right )^{-1}\sqrt{\frac{\eps_1}{\eps_0}} u,  \sqrt{\frac{\eps_1}{\eps_0}} u \rrangle_{L^2(\Omega)}
  \nonumber\\
  &= &\frac{\la}{(2\pi)^2} \int_{-\pi}^\pi\int_{-\pi}^\pi \sum_{s \in \N} (\la_s({\bf k})-\la)^{-1} \left| \llangle  \sqrt{\frac{\eps_1}{\eps_0}}u,\psi_{s}(\cdot,{\bf k})\rrangle_{ \eps_0}   \right|^2 d{\bf k}. \label{form}
  \end{eqnarray}

By continuity of the band function $\la_s$ we have, for each $s \in \N$, either   $\la_s({\bf k}) \leq \mu_0$   for all ${\bf k} \in   [\pi,\pi]^2 $, implying $(\la_s({\bf k})-\la)^{-1}  \leq 0$, or
$\la_s({\bf k})\geq \mu_1$ for all ${\bf k} \in  [\pi,\pi]^2 $. In the second case,  we either have $s \in S$, or, using Lemma
\ref{lem:5},  $\la_s({\bf k})-\mu_1$ is positively bounded away from $0$, implying $0\leq(\la_s({\bf k})-\la)^{-1} \leq C$ with $C$ independent of $\la, {\bf k}$ and $s$.

  Therefore, using Bessel's inequality, we have
  \bea
\llangle \eps_0 A_\la u,u \rrangle_{L^2((0,1)^3)}  &\leq&
\frac{\la}{(2\pi)^2} \int_{-\pi}^\pi\int_{-\pi}^\pi \sum_{s \in S} (\la_s({\bf k})-\la)^{-1} \left| \llangle  \sqrt{\frac{\eps_1}{\eps_0}}u,\psi_{s}(\cdot,{\bf k})\rrangle_{\eps_0}  \right|^2 d{\bf k}\\
&&\hspace{50pt}+  \la C\norm{\frac{\eps_1}{\eps_0}}_\infty \norm{u}^2_{\eps_0}.
   \eea

We next split the domain of integration into balls of radius $\delta$ around the points ${\bf k}^j$ and the complement of the union of these balls in $[-\pi,\pi]^2$, where $\delta$ is chosen sufficiently small so that the balls do not overlap and such  that it satisfies some additional conditions mentioned below. Then

\bea
&&\frac{\la}{(2\pi)^2} \int_{-\pi}^\pi\int_{-\pi}^\pi \sum_{s \in S} (\la_s({\bf k})-\la)^{-1} \left| \llangle  \sqrt{\frac{\eps_1}{\eps_0}}u,\psi_{s}(\cdot,{\bf k})\rrangle_{\eps_0} \right|^2 d{\bf k} \\
&=&  \frac{\la}{(2\pi)^2} \sum_{s\in S} \left[\sum_{
\stackrel {j=1}{ s_j=s} }
^n
 \iint_{B_\delta({\bf k}^j)}
 ( \la_s({\bf k})-\la)^{-1}  \left| \llangle  \sqrt{\frac{\eps_1}{\eps_0}}u,\psi_{s}(\cdot,{\bf k})\rrangle_{\eps_0}  \right|^2 d{\bf k}  \right.   \\
&&  + \left. \iint_{R_s}
  ( \la_s({\bf k})-\la)^{-1}  \left| \llangle  \sqrt{\frac{\eps_1}{\eps_0}}u,\psi_{s}(\cdot,{\bf k})\rrangle_{\eps_0}  \right|^2 d{\bf k}
   \right]
\eea
where  $R_s:=[-\pi,\pi]^2 \backslash \cup_{\stackrel{j=1}{ s_j=s}}^{\ n} B_\delta({\bf k}^j) $. On $R_s$  we again use that $(\la_s({\bf k})-\la)^{-1}$
is uniformly bounded (with respect to $\la$ and $k$), since the continuous function $\la_s(\cdot)-\mu_1$  is positive and therefore positively bounded away from $0$ on the compact set $R_s$. Moreover $S$ is finite, so
\bea
\sum_{s\in S}\iint_{R_s}
  ( \la_s({\bf k})-\la)^{-1}  \left| \llangle  \sqrt{\frac{\eps_1}{\eps_0}}u,\psi_{s}(\cdot,{\bf k})\rrangle_{\eps_0}  \right|^2 d{\bf k}
  &\leq & C\norm{\frac{\eps_1}{\eps_0}}_\infty \norm{u}^2_{\eps_0}.
  \eea
 It remains to estimate the sum of the integrals over $B_\delta({\bf k}^j)$. We next exchange the order of the sums  which can only add positive terms (if $s\in S_{ {\bf k}^j} $  for several $j$) and then shift the integration variable:
\bea
&&\frac{\la}{(2\pi)^2} \sum_{s\in S} \sum_{j=1, s_j=s}^n
 \iint_{B_\delta({\bf k}^j)}
 ( \la_s({\bf k})-\la)^{-1}  \left| \llangle  \sqrt{\frac{\eps_1}{\eps_0}}u,\psi_{s}(\cdot,{\bf k})\rrangle_{\eps_0} \right|^2 d{\bf k}\\
 &\leq& \frac{\la}{(2\pi)^2} \sum_{j=1}^n   \sum_{s\in S_{{\bf k}^j} }
 \iint_{B_\delta({\bf k}^j)}
 ( \la_s({\bf k})-\la)^{-1}  \left| \llangle  \sqrt{\frac{\eps_1}{\eps_0}}u,\psi_{s}(\cdot,{\bf k})\rrangle_{\eps_0} \right|^2 d{\bf k}\\
 &=& \frac{\la}{(2\pi)^2}   \sum_{j=1}^n   \sum_{s\in S_{{\bf k}^j} }
 \iint_{B_\delta(0)}
( \la_s({\bf k}^j+\widetilde {\bf k})-\la)^{-1}
   \left| \llangle  \sqrt{\frac{\eps_1}{\eps_0}}u,\psi_{s}(\cdot,{\bf k}^j+\widetilde {\bf k})\rrangle_{\eps_0} \right|^2 d \widetilde {\bf k}    \\
    &\leq& \frac{\la}{(2\pi)^2}    \iint_{B_\delta(0)}   (  \mu_1+\alpha |\widetilde {\bf k} |^2 -\la)^{-1}
 f(  {\bf \widetilde k},u)
    d \widetilde {\bf k}  ,
\eea
with $f(  {\bf \widetilde k},u)$ as in \eqref{star2}.
In the final estimate we assume $\delta \leq \min \{ \delta_{(\hat s, \hat k)}:   (\hat s, \hat k) \in \Sigma \}$   and $\alpha:= \min \{ \alpha_{(\hat s, \hat k)}:   (\hat s, \hat k) \in \Sigma \}$ (see Assumption \ref{as1}).
Using the Lipschitz continuity of $f(  {\bf \widetilde k},u)$,  by Corollary \ref{cor}, we can find $M>0$ such that
$ f(  {\bf \widetilde k},u)\leq f(0,u)+M|\widetilde {\bf k} |\norm{\eps_1}_\infty \norm{u}_{\eps_0}^2$  for $|\widetilde {\bf k} |<\delta$.
Let 
$$L=\left\{ u\in L^2_{\eps_0}((0,1)^3): \llangle  \sqrt{\frac{\eps_1}{\eps_0}}u,\psi_{s_i}(\cdot,{\bf k}^i)\rrangle_{\eps_0} = 0 \quad \hbox{ for  all } i=1,...,n\right\}.$$
Since $L$ is the orthogonal complement of the span of the $n$ linearly independent functions
$\sqrt{\dfrac{\eps_1}{\eps_0}}\psi_{s_i}(\cdot, {\bf k}^i)$ (see Proposition \ref{lindep})
we get that $\codim L=n$. For $u\in L$  and
$|{\bf \widetilde k}| <\delta$ we now have
$f(0,u)=0$ and thus $f({\bf \widetilde k,u})\leq M|\widetilde {\bf k} | \norm{\eps_1}_\infty \norm{u}_{\eps_0}^2$ and so
\bea
\llangle \eps_0 A_\la u,u \rrangle &\leq&
 \frac{\la}{(2\pi)^2}    \iint_{B_\delta(0)}   (  \mu_1+\alpha |\widetilde {\bf k} |^2 -\la)^{-1}
  M|\widetilde {\bf k} |\norm{\eps_1}_\infty \norm{u}_{\eps_0}^2    d \widetilde {\bf k}  +\lambda    C  \norm{\frac{\eps_1}{\eps_0}}_\infty\norm{u}^2_{\eps_0}\\
     &\leq& \dfrac{\mu_1}{(2\pi)^2}  M\norm{\eps_1}_\infty} \norm{u}^2_{\eps_0} \iint_{B_\delta(0)} \dfrac1{\alpha |{\bf \tilde k}|  }{d {\bf \tilde k}+\mu_1 C  \norm{\frac{\eps_1}{\eps_0}}_\infty\norm{u}^2_{\eps_0}\\  
    &\leq& K_{\alpha,\delta} \norm{\eps_1}_\infty\norm{u}_{\eps_0}^2,
\eea
where $K_{\alpha,\delta}$ is independent of $\la \in  (\mu_0,\mu_1)$.
Therefore,  if $\norm{\eps_1}_\infty<c:=K_{\alpha,\delta}^{-1}$, we have that $\frac{\llangle \eps_0 A_\la u,u \rrangle}{\norm{u}_{\eps_0}^2}\leq K_{\alpha, \delta} \norm{\eps_1}_\infty<1  
$ for all $u\in L$, a space of codimension $n$. By the variational characterisation (\ref{varchar}) of the eigenvalues, $\kappa_{n+1}(\la)<1$,
for all $\la \in (\mu_0,\mu_1)$. So at most the $n$ eigenvalue curves $\kappa_1(\la)$,..., $\kappa_n(\la)$ can cross $1$, and by monotonicity they can each cross at most once. Therefore,  
 no more than $n$ eigenvalues of the operator $L(k_x)$ are created in the gap.
\end{proof}

\begin{remark}
Compared to our previous non-accumulation result in \cite{BHPW12}, this result is stronger in that it gives an upper estimate on the number of eigenvalues created. However, it requires a stronger assumption on the regularity of the band function (Assumption \ref{as1}) and a sign assumption on $\eps_1$.
\end{remark}

Our next aim is to show that at least $n$ eigenvalues are  created. We first need to make an extra assumption on the regularity of the band functions.

\begin{assumption} \label{as2}
There exist $\beta,\delta>0$ 
such that $\la_{s_j}({\bf k})\leq \mu_1+\beta |{\bf k}-{\bf k}^j|^2$ for all  ${\bf k}\in B_\delta({\bf k}^j)$ and all $j=1,...,n$.
\end{assumption}


\begin{theorem} \label{lower}  For any sufficiently small $\eps_1\geq0$,   at least $n$ eigenvalues are created in the spectral gap.
\end{theorem}

\begin{proof}
Let $s'$ be such that $\mu_1$ is the lowest point of the $s'$-band and $\mu_0$ is the highest point of the $(s'-1)$-band.
Here $s'=1$ if the semi-infinite gap $(-\infty, \mu_1)$ is under consideration. We note that such an $s'$ must exist for there to be a gap. Let $\lambda\in(\mu_0,\mu_1)$.
Then since $\la_s({\bf k})-\la<0$ for $s<s'$  and all $k \in [-\pi,\pi]^2,$ we have using \eqref{form} that
\ben \label{Rayleighup}
\llangle \eps_0A_\la u,u \rrangle_{L^2({(0,1)^3})} &\leq & \frac{\la}{(2\pi)^2}\int_{-\pi}^\pi\int_{-\pi}^\pi \sum_{s \geq s'} (\la_s({\bf k})-\la)^{-1} \left| \llangle  \sqrt{\frac{\eps_1}{\eps_0}}u,\psi_{s}(\cdot,{\bf k})\rrangle_{\eps_0}  \right|^2 d{\bf k}\nonumber\\ \nonumber
&\leq& \frac{\la}{(2\pi)^2(\mu_1-\la)} \int_{-\pi}^\pi\int_{-\pi}^\pi \sum_{s \geq s'}  \left| \llangle  \sqrt{\frac{\eps_1}{\eps_0}}u,\psi_{s}(\cdot,{\bf k})\rrangle_{\eps_0}  \right|^2 d{\bf k}\\ \nonumber
&\leq& \frac{\la}{(2\pi)^2(\mu_1-\la)}\int_{-\pi}^\pi \int_{-\pi}^\pi \sum_{s\in\N}  \left| \llangle  \sqrt{\frac{\eps_1}{\eps_0}}u,\psi_{s}(\cdot,{\bf k})\rrangle_{\eps_0}  \right|^2 d{\bf k}\\
&=& \frac{\la}{\mu_1-\la} \norm{  \sqrt{\frac{\eps_1}{\eps_0}}u}_{\eps_0}^2 \ \leq\ \frac{\la\norm{\eps_1}_\infty}{(\mu_1-\la)\inf\eps_0} \norm{u}_{\eps_0}^2 .
\een
Therefore, if the perturbation $\eps_1$ is sufficiently small, we can find $\lambda'\in(\mu_0,\mu_1)$ such that
\beq
\label{upperkappa} \kappa_{1}(\la')=\sup_{\norm{u}\neq 0} \frac{ \llangle A_{\la'} u ,u \rrangle_{\eps_0} }{\norm{u }_{\eps_0}^2} < 1.
\eeq
We next give a lower bound on the Rayleigh quotient.  Let the balls $B_\delta({\bf k}^j)$ and the complement $R_s$ be chosen as in the proof of Theorem
\ref{upper}.

We again start with equality \eqref{form} and
 split the sum over $s\in \N$ into three parts:
One over $s<s'$, one over $s\geq s'$ with $s\not\in S$, and one over $s\in S$. (Note that $s\geq s'$ for all $s \in S$).
The first sum is bounded from below by  $-C\norm{u}^2_{\eps_0}$  as long as $\la$ stays away from $\mu_0$, say $\la\in (\tfrac12 (\mu_0+\mu_1),\mu_1)$.  The second sum is bounded from below by $0$.
Therefore,
\bea
&&\llangle \eps_0 A_\la u,u \rrangle_{L^2((0,1)^3)} \\&\geq&
  \frac{\la}{(2\pi)^2}  \sum_{s\in S}  \iint_{[-\pi,\pi]^2}
    (\la_s({\bf k})-\la)^{-1}  \left| \llangle  \sqrt{\frac{\eps_1}{\eps_0}}u,\psi_{s}(\cdot,{\bf k})\rrangle_{\eps_0} \right|^2 d{\bf k}
    -C\norm{u}^2_{\eps_0}
  \\
   &=&   \frac{\la}{(2\pi)^2} \sum_{s\in S} \left[\sum_{\stackrel{j=1}{s_j=s}}^n
     \iint_{B_{\delta}({\bf  k^j})}
      ( \la_s({\bf k})-\la)^{-1}  \left| \llangle  \sqrt{\frac{\eps_1}{\eps_0}}u,\psi_{s}(\cdot,{\bf k})\rrangle_{\eps_0} \right|^2 d{\bf k}  \right.   \\
   &&
 +  \iint_{R_s}
 \left .   ( \la_s({\bf k})-\la)^{-1}  \left| \llangle  \sqrt{\frac{\eps_1}{\eps_0}}u,\psi_{s}(\cdot,{\bf k})\rrangle_{\eps_0}  \right|^2 d{\bf k}\right]
-C \norm{u}^2_{\eps_0}\\
   &\geq&   \frac{\la}{(2\pi)^2} \sum_{s\in S} \sum_{\stackrel{j=1}{s_j=s}}^n
     \iint_{B_{\delta}({\bf  k^j})}
      ( \la_s({\bf k})-\la)^{-1}  \left| \llangle  \sqrt{\frac{\eps_1}{\eps_0}}u,\psi_{s}(\cdot,{\bf k})\rrangle_{\eps_0}   \right|^2 d{\bf k} 
-C \norm{u}^2_{\eps_0}\\
  &\geq&\frac{\la}{(2\pi)^2n}\sum_{j=1}^n  \sum_{s\in S_{{\bf k}^j}}
  \iint_{B_\delta(0)}  ( \la_s({\bf k}^j+\widetilde {\bf k})-\la)^{-1}  \left| \llangle  \sqrt{\frac{\eps_1}{\eps_0}}u,\psi_{s}(\cdot,{\bf k}^j+\widetilde {\bf k}) \rrangle_{\eps_0}   \right|^2 d\widetilde {\bf k}  -C\norm{u}_{\eps_0}^2
     \\
  &\geq&   \frac{\la}{(2\pi)^2n}
  \iint_{B_\delta(0)}
 (\mu_1+\beta |\widetilde {\bf k}|^2-\la)^{-1}
 f(  {\bf \widetilde k},u) d\widetilde {\bf k} -C\norm{u}_{\eps_0}^2,
\eea
using  the fact that any $s\in S$ can be at most in $n$ sets $S_{{\bf k}^j}$ in the last but one inequality, and \eqref{star2} in the last line.
For any function
\begin{equation}
u=\sum_{i=1}^n \xi_i \sqrt{  \frac{\eps_0}{\eps_1}}  \psi_{s_i}(\cdot , {\bf k}^i)\chi_D
\label{label}
\end{equation}
with coefficients $(\xi_i)_{i=1}^n\in\C^n$, we have, setting $G_{ij}:=\llangle \psi_{s_i}(\cdot, {\bf k}^i), \psi_{s_j}(\cdot,{\bf k}^j))\rrangle_{L^2(D)}$, 
\bea
f(0,u) &=&\sum_{j=1}^n \sum_{s\in S_{{\bf k}^j}}
\left|  \sum_{i=1}^n \xi_i
\llangle \psi_{s_i}(\cdot, {\bf k}^i), \psi_s(\cdot,{\bf k}^j))\rrangle_{L^2(D)} \right|^2 \\
&\geq& \sum_{j=1}^n \left| \sum_{i=1}^n \xi_i  G_{ij}\right|^2 = |G\xi|^2 \geq c |\xi|^2,
\eea
as $G=(G_{ij})$ is the Gram matrix of a linearly independent set of vectors (see Proposition \ref{lindep})
and therefore it is invertible and $G^*G$ is strictly positive.
  Furthermore, since $\eps_1$ is positively bounded  away from $0$ on $D$, (\ref{label}) implies $| \xi |^2 \geq \widetilde{c} \norm{u}^2_{\eps_0}$. Now Lipschitz continuity of $f(\widetilde  {\bf k},u)$ (Corollary \ref{cor}) implies that
\bea
f(\widetilde  {\bf k},u) &\geq& f(0,u) -  |\widetilde  {\bf k}| M \norm{\eps_1}_\infty\norm{u}_{\eps_0}^2\\
&\geq & c \widetilde{c} \norm{u}^2_{\eps_0} - |\widetilde  {\bf k}| M \norm{\eps_1}_\infty\norm{u}_{\eps_0}^2\\
&\geq & \left( c \widetilde{c} - \delta M \norm{\eps_1}_\infty\right)\norm{u}_{\eps_0}^2\ \geq \ C \norm{u}_{\eps_0}^2
\eea 
for some $C>0$, provided $ |\widetilde  {\bf k}|<\delta$ and $\delta$ is sufficiently small.
To show that the Rayleigh quotient becomes arbitrarily large as $\la\to\mu_1$, it is therefore sufficient for $$\iint\limits_{B_\delta(0)}\frac{1}{ \mu_1+\beta|\widetilde{\bf k}|^2-\la} d\widetilde{\bf k}$$ to diverge in the limit as $\la\nearrow\mu_1$.

Changing to polar coordinates, we have
\ben
\iint\limits_{B_\delta(0)}\frac{1}{ \mu_1+\beta|\widetilde{\bf k}|^2-\la} d\widetilde{\bf k} &=& \int_0^{2\pi}\int_0^\delta \frac{r\ dr\ d\theta}{\mu_1-\la+\beta r^2}  \nonumber\\ \label{polar}
 &=& \frac{\pi}{\beta} \ln (\mu_1-\la+\beta r^2)\big\vert_0^\delta\\
 \nonumber &=& \frac{\pi}{\beta} \left[ \ln (\mu_1-\la+\beta \delta^2) - \ln(\mu_1-\la)\right] \to \infty \hbox{ as } \la\nearrow\mu_1.
\een
Therefore, $\min  \dfrac{ \llangle  \eps_0 A_{\lambda }u,u \rrangle}{ \norm{u}^2_{\eps_0}} \to \infty $ as $\la \nearrow \mu_1$
where the minimum is taken over the $n$-dimensional subspace given by functions defined as in (\ref{label}). Hence, by \eqref{varchar}, we have $\kappa_n(\la) \nearrow \infty$ as $\la \to \mu_1$, and combined with Lemma \ref{lem:cont} and  \eqref{upperkappa} this means that at least $n$ eigenvalues are   created in the gap.
\end{proof}

\begin{remark}
The divergence of the intergal in \eqref{polar} depends on the space dimension and prevents generalising the results to higher dimensions.
\end{remark} 

Combining Theorems \ref{upper}  and Theorem \ref{lower} yields the following result.
\begin{theorem}  \label{theoremall}  Let $\eps_1\geq 0$ be sufficiently small. Then the number of eigenvalues of the operator $L(k_x)$ in the gap $(\mu_0,\mu_1)$ equals $n$, the number of solution pairs
 $(s,\bf{k})$ of the equation $\mu_1=\lambda_s(\bf{k})$.
\end{theorem}


We end this paper with a result for non-negative perturbations of an arbitrary size.
\begin{theorem}
For any non-negative perturbation $\eps_1$ the number of eigenvalues of the operator $L(k_x)$ in the gap $(\mu_0,\mu_1)$ is finite.
\end{theorem}
\begin{proof}
Define the space $L$ as in the proof of Theorem \ref{upper} and let $P_L$ denote the orthogonal projection from $ L^2_{\eps_0}((0,1)^3)$ onto $L$. Consider the operator
\beq\label{Bt}\Bt u:=\sum_{s\in \N}\iint_{[-\pi,\pi]^2}  (\la_s({\bf k})-\mu_1)^{-1}\llangle \sqrt{\frac{\eps_1}{\eps_0}}P_Lu,\psi_s(\cdot,{\bf k})\rrangle_{\eps_0} \psi_s(\cdot,{\bf k}) \ d{\bf k}.\eeq
Although it is possible for $\la_s({\bf k})-\mu_1$ to vanish for certain $(s,{\bf k})$, we show that
$\Bt$ is a well defined operator. To this end, we estimate the $H^1$-norm of $\Bt u$.
As before, we split the sum in \eqref{Bt} into a sum over $s\in S$, denoted by $\Bt_1u$ and a sum over $s\not \in S$, denoted
by $\Bt_2u$.
Consider the $H^1$-norm of  $\Bt_1u$, i.e.
\bea
\norm{ \Bt_1u}_{H^1}=
\left\|\sum_{s\in S}
\iint _{[-\pi, \pi]^
2}(\la_s({\bf k})-\mu_1)^{-1}\llangle \sqrt{\frac{\eps_1}{\eps_0}}P_Lu,\psi_s(\cdot,{\bf k})\rrangle_{\eps_0} \psi_s(\cdot,{\bf k}) \ d{\bf k}
\right\|_{H^1(0, 1)^3}.
\eea
We take the $H^1$-norm inside the integrals on the right hand side to get an upper bound, and it remains to estimate the norm of the integrand.
Again, we split the domain of integration into a union of balls $B_\delta(\bk^j)$ together with the complement.
The integrals over the complement can be estimated in a straightforward manner by $C\|u\|_{L^2((0, 1)^3)}$
using that $(\lambda_s(\bk)-\mu_1)$ is bounded away from zero and $\|\psi_s(\cdot, \bk)\|_{H^1(0, 1)^3}=\sqrt{1+\lambda_s(\bk)}$.
The remaining part can be estimated as
\ben
  &&   \sum_{s\in S} \sum_{\stackrel{j=1}{s_j=s}}^{n}
  \iint_{B_\delta(\bk^j)}   (\la_s({\bf k})-\mu_1)^{-1}  \left| \llangle  \sqrt{\frac{\eps_1}{\eps_0}}P_L u,\psi_{s}(\cdot,{\bf k})\rrangle_{\eps_0}  \right|\
  \|\psi_{s}(\cdot, \bk)\|_{H^1(0, 1)^2} d{\bf k}
  \nonumber    \\
  &\leq& C \sum_{s\in S} \sum_{\stackrel{j=1}{s_j=s}}^{n}
  \iint_{B_\delta(0)}
 ( \la_s({\bf k}^j+\widetilde {\bf k})-\mu_1)^{-1}
    \left| \llangle  \sqrt{\frac{\eps_1}{\eps_0}}P_L u,\psi_{s}(\cdot,{\bf k}^j+\widetilde {\bf k})\rrangle_{\eps_0}  \right| d \widetilde {\bf k}    \nonumber\\
     &\leq&   C \iint_{B_\delta(0)}   (  \mu_1+\alpha |\widetilde {\bf k} |^2 -\mu_1)^{-1}
  \sum_{j=1}^n \sum_{s\in S_{{\bf k}^j}}
   \left| \llangle  \sqrt{\frac{\eps_1}{\eps_0}}P_L u,\psi_{s}(\cdot,{\bf k}^j+\widetilde {\bf k})\rrangle_{\eps_0}  \right|
     d \widetilde {\bf k},  \label{aboveinequality}
   \een
   where $\alpha:= \min \{ \alpha_{(\hat s, \hat k)}:   (\hat s, \hat k) \in \Sigma \}$ (see Assumption \ref{as1}).
The double sum in the above expression can be estimated by $ f(\widetilde {\bf k},P_L u)^{1/2}$, as can be seen by applying
the Cauchy-Schwarz inequality. Using the Lipschitz continuity of $f(\widetilde {\bf k},P_L u)$,   and the fact that      $f(0,P_L u)=0$,
we can find $M>0$ such that
$ f(\widetilde {\bf k},P_L u)\leq f(0,P_L u)+M|\widetilde {\bf k} | \norm{P_L u}^2_{\eps_0}=M|\widetilde {\bf k} | \norm{P_L u}^2_{\eps_0}$  for $|\widetilde {\bf k} |<\delta$.
Writing the integral in terms of polar coordinates as in \eqref{polar}, it is clear that it converges and so we get
$\norm{\Bt_1 u}_{H^1}\leq C\norm{u}_{\eps_0}$.

Next, we estimate $\Bt_2$.
The inverse Floquet-Bloch transform has the following mapping properties:  $$U^{-1}: L^2( (-\pi,\pi)^2,\; L^2((0,1)^3)) \to L^2(\Omega),\quad g\mapsto \dfrac1{2\pi}\int_{[-\pi, \pi]^2} g(\cdot,{\bf k})d{\bf k},$$ where each  $g(\cdot,{\bf k})$ is quasi-periodically extended. Moreover,
$U^{-1}$ is a topological isomorphism, also between    $L^2((-\pi, \pi)^2, H^1((0, 1)^3))$   and   $H^1(\Omega)$. Therefore,
 we have 
\bea
\norm{\tilde B_2u}_{H^1(( 0,1)^3)}^2&=&\norm{\sum_{s\not \in S}\iint_{[-\pi,\pi]^2} (\la_s({\bf k})-\mu_1)^{-1}\llangle \sqrt{\frac{\eps_1}{\eps_0}} P_Lu,\psi_s(\cdot,{\bf k})\rrangle_{\eps_0} \psi_s(\cdot,{\bf k}) \ d{\bf k}}_{H^1(( 0,1)^3)}^2 \\
&\leq& \norm{\sum_{s\not \in S}\iint_{[-\pi,\pi]^2} (\la_s({\bf k})-\mu_1)^{-1}\llangle \sqrt{\frac{\eps_1}{\eps_0}} P_Lu,\psi_s(\cdot,{\bf k})\rrangle_{\eps_0} \psi_s(\cdot,{\bf k}) \ d{\bf k}}_{H^1( \Omega)}^2 \\
&\leq&   C \iint_{[-\pi,\pi]^2} \norm{ \sum_{s\not \in S} ( \la_{s}({\bf k})-\mu_1)^{-1}\llangle \sqrt{\frac{\eps_1}{\eps_0}} P_Lu,\psi_s(\cdot,{\bf k})\rrangle_{\eps_0} \psi_s(\cdot,{\bf k})}_{H^1( 0,1)^3}^2  \hspace{-5pt} d{\bf k}.
\eea
 Set
$$
Q_S({\bf k}) := \sum _{s\not \in S}  \langle \cdot  , \psi_s(\cdot, {\bf k}) \rangle \psi_s(\cdot,{\bf k}).
$$
and note that by definition of the Bloch functions, we have for each $\bk$,
\ben\label{defG} &
\sum_{s\not \in S} \left(\la_{s}({\bf k})-\mu_1\right)^{-1} \llangle \sqrt{\frac{\eps_1}{\eps_0}} P_Lu, \psi_s(\cdot, \bk) \rrangle \psi_s(\cdot,\bk)\\ 
&=\left( L_0(k_x,{\bf k})-\mu_1\right)^{-1}  Q_S ({\bf k}) \left[\sqrt{ \dfrac{\eps_1}{\eps_0}} P_L u\right]\nonumber
\een
since $Q_S(\bk)$ projects exactly onto the Bloch waves $\psi_s(\cdot, \bk)$ with $s\not\in S$.  We denote by $R(\bk)$ for the right-hand side of \eqref{defG}. From the equality \eqref{defG}, we
immediately get
$$
\|R(\bk)\|_{L^2((0, 1)^3)} \leq C \|u\|_{L^2((0,1)^3)}
$$
with $C > 0$ and independent of $\bk$.
On the other hand, 
$$\left(L_0(k_x,{\bf k})-\mu_1\right) R(\bk) = 
Q_S ({\bf k}) \left[\sqrt{ \dfrac{\eps_1}{\eps_0}} P_L u\right],$$
 and hence
by applying standard regularity estimates \cite{Agmon}, we get
\bea
\|R(\bk)\|_{H^1((0, 1)^3)} &\leq& C \|R(\bk)\|_{L^2((0, 1)^3)}+
C \left\|Q_S ({\bf k}) \left[\sqrt{ \dfrac{\eps_1}{\eps_0}} P_L u\right] \right\|_{L^2((0, 1)^3)}\\
&\leq& C \|u\|_{L^2((0, 1)^3)}.
\eea
Together, this proves that the operator $\Bt :L^2_{\eps_0}((0,1)^3)\to L^2_{\eps_0}((0,1)^3)$ defined in \eqref{Bt} is compact, in particular, it is a well-defined linear operator on $L$.

Define the operator $B$ on $L$ by
\be
 Bu=\dfrac1{(2\pi)^2} \mu_1 P_L\sqrt{\frac{\eps_1}{\eps_0}}\Bt u.
\ee
Then $B:L\to L$ is compact and can only have finitely many eigenvalues greater than $1$.

Since for $\la\in(\mu_0,\mu_1)$ and for all $(s,{\bf k})$ such that $\la_s({\bf k})\neq\mu_1$ we have
$$\frac{\mu_1}{\la_s({\bf k})-\mu_1}-\frac\la{\la_s({\bf k})-\la}>0,$$
 (\ref{form}) implies that,   for any $u\in L^2_{\eps_0}((0,1)^3)$,
  \begin{eqnarray*}
  &&\llangle \eps_0A_\la P_L u,P_L u \rrangle_{L^2({(0,1)^3})}\\
  &= &\frac{\la}{(2\pi)^2} \int_{-\pi}^\pi\int_{-\pi}^\pi \sum_{s \in \N} (\la_s({\bf k})-\la)^{-1} \left| \llangle  \sqrt{\frac{\eps_1}{\eps_0}}P_L u,\psi_{s}(\cdot,{\bf k})\rrangle_{L^2_{\eps_0}({(0,1)^3})}  \right|^2 d{\bf k}\\
  &\leq & \frac{\mu_1}{(2\pi)^2} \int_{-\pi}^\pi\int_{-\pi}^\pi \sum_{s \in \N} (\la_s({\bf k})-\mu_1)^{-1} \left| \llangle  \sqrt{\frac{\eps_1}{\eps_0}}P_L u,\psi_{s}(\cdot,{\bf k})\rrangle_{L^2_{\eps_0}({(0,1)^3})}  \right|^2 d{\bf k}\\
  &=&  \llangle \eps_0 B P_L u,P_L u \rrangle_{L^2({(0,1)^3})}<\infty.
  \end{eqnarray*}
 The expression is finite, as all sums and integrals converge by a similar argument to the one used in (\ref{aboveinequality}).
Let $l>\codim L$ and consider decompositions of $L$ of the form $L= M\bigoplus N$ with $\dim M = l-1-\codim L$. Then the co-dimension of the space $N$ in  $L^2_{\eps_0}((0,1)^3)$ is  $l-1$, while its co-dimension in $L$ is $l-1-\codim L$. By the variational characterisation (\ref{varchar}) of the eigenvalues, we then get
\bea
\kappa_l(\la) &=& \inf_{codim X=l-1} \sup_{u\in X} \frac{\llangle A_\la u, u\rrangle_{\eps_0}}{\llangle u, u\rrangle_{\eps_0}}\\
&\leq& \inf_{N} \sup_{u\in N} \frac{\llangle A_\la P_L u,P_L u\rrangle_{\eps_0}}{\llangle u, u\rrangle_{\eps_0}}\\
&\leq& \inf_{N} \sup_{u\in N} \frac{\llangle B u, u\rrangle_{\eps_0}}{\llangle u, u\rrangle_{\eps_0}} = \kappa_{l-\codim L}(B).
\eea
Therefore,  the number of eigenvalues of $A_\la$ that can be greater than $1$ is uniformly bounded for all $\la \in (\mu_0,\mu_1)$. This implies the result by monotonicity of the eigenvalues (see Lemma \ref{lem:cont}).
\end{proof}

\section*{Acknowledgement}
 The authors would like to thank the Isaac Newton Institute, where much of this work was undertaken.

\end{document}